\documentclass[a4paper,11pt]{scrartcl}
\addtokomafont{disposition}{\rmfamily}

\usepackage[left=2.5cm,right=2.5cm,top=2.5cm,bottom=2.5cm]{geometry}

\usepackage{amsthm}
\usepackage{amsmath}
\usepackage{amssymb}
\usepackage{mathrsfs} 
\usepackage{graphicx}

\usepackage[dvipsnames]{xcolor}
\usepackage{stackengine}
\usepackage{placeins}
\usepackage{arydshln}
\usepackage{bbm}
\usepackage{textcomp}
\usepackage{enumitem}
\usepackage{mathtools}
\usepackage{tabularx}
\usepackage{booktabs}
\usepackage{rotating}

\usepackage{color}
\usepackage{framed}
\usepackage{comment}
\definecolor{shadecolor}{gray}{0.9}

\usepackage{dsfont} 
\usepackage{caption} 
\usepackage{subcaption} 
\usepackage{graphicx}
\usepackage{pdfpages}
\usepackage{url}
\usepackage[english]{babel}
\usepackage{setspace}
\usepackage{dashrule}
\usepackage{tikz}
\usetikzlibrary{automata, positioning, arrows}
\tikzset{
        ->,  
        node distance=5.5cm, 
        every state/.style={thick, fill=gray!10}, 
        initial text=$ $, 
        }

\usepackage{xr-hyper} 
\usepackage[
breaklinks,
colorlinks=true,
pagebackref=false,
pdfauthor={},%
pdftitle={},%
citecolor=blue,
linkcolor=blue,
urlcolor=blue,
filecolor=blue
]{hyperref}      
\usepackage{cleveref}
\usepackage[authoryear,round]{natbib}

\specialcomment{extra}{\begin{shaded}}{\end{shaded}}

\theoremstyle{plain}  
\newtheorem{theorem}{Theorem}
 
\newtheorem{proposition}{Proposition} 
 
\newtheorem{assumption}{Assumption}

\theoremstyle{definition} 
\newtheorem{definition}{Definition}
 
\newtheorem{example}{Example}
\newtheorem{rem}[theorem]{Remark}

\theoremstyle{remark}

\newcommand{\diff}{\mathrm{d}}
\newcommand{\dint}{\,\mathrm{d}}
\newcommand{\E}{\mathbb{E}}

\newcommand{\eps}{\varepsilon}

\newcommand{\F}{\mathcal{F}}
\newcommand{\R}{\mathbb{R}}

\newcommand{\one}{\mathds{1}}
\newcommand{\interior}{\operatorname{int}}

\newcommand{\rank}{\operatorname{rank}}

\newcommand{\ph}{\varphi}

\DeclareMathOperator*{\argmin}{arg\,min}

\newcommand{\N}{\mathbb N}
\renewcommand{\P}{\mathbb P}

\newcommand{\X}{\mathcal{X}}
\newcommand{\Y}{\mathcal{Y}}
\newcommand{\Z}{\mathcal{Z}}

\newcommand{\wt}{\widetilde}

\renewcommand{\rm}{\normalfont \rmfamily}
\renewcommand{\bf}{\normalfont \bfseries}

\newcommand{\as}{\text{a.s.}}

\def\be{\begin{equation} \label}
\def\ee{\end{equation}}

\usepackage[colorinlistoftodos,textsize=tiny]{todonotes}
\newcommand{\Comments}{1}
\newcommand{\mynote}[2]{\ifnum\Comments=1\textcolor{#1}{#2}\fi}
\newcommand{\mytodo}[2]{\ifnum\Comments=1%
  \todo[linecolor=#1!80!black,backgroundcolor=#1,bordercolor=#1!80!black]{#2}\fi}

\ifnum\Comments=1               
\fi

\ifnum\Comments=1               
\fi

\ifnum\Comments=1               
\fi

\allowdisplaybreaks

\begin{document}

\title{Characterizing M-estimators}
\author{Timo Dimitriadis\thanks{Heidelberg University, Alfred Weber Institute of Economics, Bergheimer Str.\ 58, 69115 Heidelberg, Germany and
		Heidelberg Institute for Theoretical Studies, 69118 Heidelberg, Germany, e-mail: \href{mailto:timo.dimitriadis@awi.uni-heidelberg.de}{timo.dimitriadis@awi.uni-heidelberg.de}}
	\and Tobias Fissler\thanks{Vienna University of Economics and Business (WU), Department of Finance, Accounting and Statistics, Welthandelsplatz 1, 1020 Vienna, Austria, 
		e-mail: \href{mailto:tobias.fissler@wu.ac.at}{tobias.fissler@wu.ac.at}} \and Johanna Ziegel\thanks{University of Bern, Department of Mathematics and Statistics, Institute of Mathematical Statistics and Actuarial Science, Alpeneggstrasse 22, 3012 Bern, Switzerland, 
		e-mail: \href{mailto:johanna.ziegel@stat.unibe.ch}{johanna.ziegel@stat.unibe.ch}}
}
\maketitle

\begin{abstract}
	We characterize the full classes of M-estimators for semiparametric models of general functionals by formally connecting the theory of consistent loss functions from forecast evaluation with the theory of M-estimation.
	This novel characterization result opens up the possibility for theoretical research on efficient and equivariant M-estimation and, more generally, it allows to leverage existing results on loss functions known from the literature of forecast evaluation in estimation theory.
\end{abstract}


\section{Introduction}
\label{sec:intro}

\onehalfspacing

The task of regression is to model the effect of covariates $X$ on a response variable $Y$, or more precisely, the effect of $X$ on a functional $\Gamma$ of the conditional distribution of $Y$ given $X$, $F_{Y|X}$.
The typical example for $\Gamma$ is the mean, resulting in mean regression. 
In many application, one is interested in other functionals such as quantiles (Value at Risk, VaR), expectiles, variance, or Expected Shortfall (ES) \citep{Koenker1978, Bollerslev1986, Patton2019}.
A correctly specified parametric model $m(X,\theta)$ satisfies
$\Gamma(F_{Y|X}) = m(X,\theta_0)$ for some unique parameter $\theta_0\in\Theta$.

The statistician's task is to estimate the parameter $\theta_0$ based on data $(Y_t,X_t)$, $t=1, \ldots, N$. 
For the standard situation of linear mean regression, $\Gamma(F_{Y|X}) = \E[Y|X]$, and $m(X, \theta) = X^\intercal \theta$,
one often employs the ordinary-least-squares (OLS) estimator of the form
$\widehat \theta_{T} = \argmin_{\theta \in\Theta} \frac{1}{T} \sum_{t=1}^T \big(Y_t - X_t^\intercal \theta\big)^2$, or a related closed-form solution thereof.  
The OLS estimator is a special instance of an M-estimator \citep{Huber1967, NeweyMcFadden1994},
\be{eq:M-est}
\widehat \theta_{T} = \argmin_{\theta \in\Theta} \frac{1}{T} \sum_{t=1}^T \rho \big(Y_t , m(X_t,\theta)\big),
\ee
based on a loss function $\rho$, which is the key ingredient of an M-estimator.
Note that our results carry over to time-varying loss functions $\rho_t$.

The core condition on $\rho$ for consistency of $\widehat \theta_{T}$ is that 
$\E\big[\rho \big(Y_t , m(X_t,\theta_0)\big)\big]<\E\big[\rho \big(Y_t , m(X_t,\theta)\big)\big]$ for all $\theta\neq\theta_0$ and for all $t\in\N$, which we call \emph{strict model-consistency} of $\rho$ for $m$. 
Apart from special cases, the classes of such loss functions $\rho$ are unfortunately not well understood in M-estimation yet.
Our main result, Theorem \ref{thm:hierarchy consistency}, establishes that a loss function $\rho$ is (strictly) model-consistent for $m$ if and only if it is (strictly) \emph{consistent} for the target functional $\Gamma$, meaning that 
$
\int \rho \big(y,\Gamma(F)\big)\mathrm{d} F(y) \le \int \rho(y,\xi)\mathrm{d} F(y)
$
for all $\xi \in \R^k$ and for all distributions $F$ in a sufficiently rich class, where strictness means that equality implies $\xi=\Gamma(F)$.
Since there are well-understood characterization results for strictly consistent losses from the literature of forecast evaluation \citep{Gneiting2011, FisslerZiegel2016}, Theorem \ref{thm:hierarchy consistency} lifts these results to a novel characterization of the classes of consistent M-estimators for general (vector-valued) functionals.

Our result provides the first characterization of the full classes of consistent M-estimators for semiparametric models of general, possibly vector-valued functionals by formally connecting the two strands of literature on M-estimation and forecast evaluation. 
Understanding the full classes of M-estimators facilitates a deeper understanding of the (im)possibilities, e.g., in the following areas.
It allows to derive M-estimation efficiency bounds, similar to efficient generalized method of moments estimation of \citet{Chamberlain1987}.
Drawing on the literature of homogeneous and equivariant loss functions \citep{NoldeZiegel2017, FisslerZiegel2019}, our connection allows to classify such M-estimators having e.g.~the advantage that they are invariant to linear rescaling in finite samples.
Furthermore, estimators with the most beneficial integrability conditions in the sense that $\E[ \big| \rho \big(Y_t , m(X_t,\theta)\big) \big|]$ is finite can be derived.
Our result further allows to leverage recent discoveries in the forecast evaluation literature such as mixture representations \citep{EhmETAL2016} or score decompositions \citep{DGJ_2021} for the purpose of M-estimation.
Theorem \ref{thm:hierarchy consistency} also provides the first formal argument why consistent M-estimation is basically impossible if there do not exist (strictly) consistent loss functions for the target functional $\Gamma$. This argument has already been used informally for the Expected Shortfall \citep{Patton2019, DimiBayer2019}, the Range Value at Risk \citep{Barendse2020}, and the mode \citep{Kemp2012}.


\section{Notation and Definitions}
\label{subsec:Notation and Setting}

Let $(\Omega, \mathcal A, \P)$ be a non-atomic, complete probability space where all random variables are defined. 
We introduce a class $\Y$ of $\R^d$-valued possible response variables, a corresponding class $\X$ or $\R^p$-valued regressors, and $\Z$ the class of possible response--regressor pairs $(Y,X)$. 
The class of marginal distributions $F_Y$ of $Y\in\Y$ is denoted by $\F_\Y$, with a corresponding notation $\F_\X$.
$\F_{\Y|\X}$ is the class of regular versions of conditional distributions $F_{Y|X}$ for any $(Y,X)\in \Z$; see Appendix \ref{sec:App1} for technical details.
We will identify cumulative distribution functions with their corresponding measures where convenient.
Let $\Gamma \colon\F_{\Y|\X}\to\Xi\subseteq\R^k$ be some $k$-dimensional, single-valued functional of the conditional distribution of $Y$ given $X$. 
Let $\Theta\subseteq \R^q$ be some parameter space with non-empty interior, $\interior(\Theta)$, and $m\colon \R^p\times \Theta\to \Xi$ a parametric model for the functional $\Gamma$. 
	We shall work under the following assumption of a correctly specified model with a uniquely identified true model parameter.
	\begin{assumption}
		\label{ass:unique model}
		For all $Z=(Y, X)\in\Z$ there is a unique parameter $\theta_0 = \theta_0(F_Z) \in \interior(\Theta)$ such that almost surely
		\be{eq:unique model}
		m(X,\theta_0) = \Gamma(F_{Y|X}). 
		\ee
	\end{assumption}
	Assumption \ref{ass:unique model} is semiparametric in the sense that the finite-dimensional parameter $\theta_0$ in \eqref{eq:unique model} does not fully describe the distribution of $(Y,X)$, but only $\Gamma(F_{Y|X})$, the component of the conditional distribution we are essentially interested in.
	In general, an estimator should be valid on a class of random variables $\Z$ which is as large as possible, allowing to apply the estimator in many different situations and under uncertainty of the distributions of the underlying data.
	Hence, Assumption \ref{ass:unique model} is a minimal condition on the random variables that allows for semiparametric modeling of a functional of the conditional distribution.

	We continue by recalling the use of loss functions in the closely related area of forecast evaluation, where the notion of a \emph{strictly consistent} loss function is a crucial concept in the literature on forecast evaluation, since it incentivizes truthful reports  \citep{MurphyDaan1985}.
	Making use of a similar decision-theoretic terminology as in  \cite{Gneiting2011} and \cite{FisslerZiegel2016},
	let $\F$ be some generic class of probability distributions on $\R^d$, which is our {observation domain}, and $\Xi\subseteq \R^k$, our {action domain}.

	\begin{definition}[Consistency and elicitability]
		\label{defn:consistency}
		A loss function $\rho\colon \R^d\times \Xi\to\R$ is called \emph{$\F$-consistent} for a functional $\Gamma\colon \F\to\Xi$ if $\rho(\cdot, \xi)$  is $F$-integrable for all $F\in\F$ and for all $\xi\in\Xi$, and if
		\be{eq:strict consistency}
		\int \rho\big(y,\Gamma(F)\big) \dint F(y) \le \int \rho (y, \xi)\dint F(y) \qquad \text{for all } F\in\F, \ \text{for all } \xi\in\Xi\,.
		\ee
		If equality in \eqref{eq:strict consistency} implies $\xi = \Gamma(F)$, then the loss function is called \emph{strictly $\F$-consistent} for $\Gamma$.
		A functional $\Gamma\colon \F\to\Xi$ is \emph{elicitable} if there is a strictly $\F$-consistent loss function for it.
	\end{definition}
	
	On the class of distributions with a finite second moment, 
	the squared loss $\rho(y,\xi) = (y-\xi)^2$ is strictly consistent for the mean functional. 
	More generally, subject to regularity and integrability conditions on $\rho$ and richness conditions on $\F$, $\rho$ is (strictly) $\F$-consistent for the mean if and only if it is a so-called \emph{Bregman loss} 
	\begin{align}
		\label{eqn:BregmanLoss}
		\rho(y,\xi) = \phi(y) - \phi(\xi) +\phi'(\xi)(\xi-y) + \kappa(y),
	\end{align}
	where $\phi$ is a (strictly) convex function on $\R$ with subgradient $\phi'$ and $\kappa$ is any function of $y$ \citep{Savage1971, Gneiting2011}.
	Likewise, the well known pinball or asymmetric absolute loss $\rho(y,\xi) = (\one\{y\le \xi\} - \alpha)(\xi - y)$ is strictly consistent for the lower $\alpha$-quantile on the class of distributions with finite mean and where the lower $\alpha$-quantile coincides with the upper $\alpha$-quantile.
	Moreover,  subject to regularity and integrability conditions on $\rho$ and richness conditions on $\F$, a loss $\rho$ is (strictly) consistent for the lower $\alpha$-quantile, $\alpha\in(0,1)$, if and only if it is a \emph{generalized piecewise linear loss function} 
	\begin{align}
		\label{eqn:GPLLoss}
		\rho(y,\xi) = (\one\{y\le \xi\} - \alpha)(g(\xi) - g(y)) + \kappa(y),
	\end{align}
	where $g$ is (strictly) increasing \citep{Gneiting2011b}.
	
	Similar characterization results exist for other functionals such as expectiles or the pairs consisting of the mean and variance or the quantile and Expected Shortfall \citep{Gneiting2011, FisslerZiegel2016}.
	The characterization results in \eqref{eqn:BregmanLoss} and \eqref{eqn:GPLLoss} rely on the fact that the related classes $\mathcal{F}$ are convex and rich enough. 
	E.g., if we restrict attention to symmetric distributions only, the mean equals the median, and loss functions of the form given in \eqref{eqn:BregmanLoss} or \eqref{eqn:GPLLoss} would elicit the mean and median.
	
	The following definition develops similar notions of consistency for the setting of M-estimation with an underlying class $\Z$ implying that $\Gamma$ is defined on the class of conditional distributions $\F_{\Y|\X}$.
	
	\begin{definition}[Model-consistency]
		\label{defn:model-consistency}
		Suppose Assumption \ref{ass:unique model} holds for the parametric model $m\colon\R^p\times\Theta\to\Xi$ and the functional $\Gamma\colon\F_{\Y|\X}\to\Xi$.
		Let $\rho\colon \R\times \Xi\to\R$ be a loss function such that $\E\big|\rho\big(Y,m(X,\theta)\big)\big|<\infty$ for all $(Y,X) \in\Z$ and for all $\theta\in\Theta$. 
		\begin{enumerate}[label=(\roman*)]
			\item
			The loss $\rho$ is \emph{unconditionally $\F_\Z$-model-consistent} for the model $m$ if 
			\be{eq:unconditional model consistency}
			\E\big[\rho\big(Y,m(X,\theta_0)\big)\big]\le \E\big[\rho\big(Y,m(X,\theta)\big)\big]
			\qquad  \text{for all } (Y,X)\in\Z, \  \text{for all } \theta\in\Theta\,.
			\ee 
			Moreover, $\rho$ is \emph{strictly unconditionally $\F_\Z$-model-consistent} for the model $m$ if equality in \eqref{eq:unconditional model consistency} implies that $\theta = \theta_0$.
			\item
			The loss $\rho$ is \emph{conditionally $\F_\Z$-model-consistent} for the model $m$ if 
			\be{eq:conditional model consistency}
			\E\big[\rho\big(Y,m(X,\theta_0)\big)\big|X\big]\le \E\big[\rho\big(Y,m(X,\theta)\big)\big|X\big]\ \as
			\quad \text{for all }  (Y,X)\in\Z, \  \text{for all }  \theta\in\Theta\,.
			\ee
			Moreover, $\rho$ is \emph{strictly conditionally $\F_\Z$-model-consistent} for the model $m$ if almost sure equality in \eqref{eq:conditional model consistency} implies that $\theta=\theta_0$.
		\end{enumerate}
	\end{definition}
	
	
	The concept of unconditional model-consistency is the central condition for consistent M-estimation: \citet[Property 3.3 and 3.4]{Gourieroux1987} show equivalence of these two notions in terms of first order conditions and under some regularity assumptions.
	Also see condition (i) in \citet[Theorem 2.1]{NeweyMcFadden1994}, \citet[Assumption (A-4)]{Huber1967}, where their remaining assumptions are merely regularity conditions.
	In contrast, the conditional notion is appealing since it bridges the gap between consistency for $\Gamma$ according to Definition \ref{defn:consistency} and unconditional model-consistency for $m$ in the proof of Theorem \ref{thm:hierarchy consistency} below. 	It can still be practically useful when we resort to nonparametric kernel regressions or in the presence of repeated observations of $X$, e.g.\ if $X$ consists of categorical variables only.
	
	While the classical notion of consistency in \eqref{eq:strict consistency} and the unconditional model version in \eqref{eq:unconditional model consistency} are closely related, they are crucially different in that in the latter, the expectation is also taken with respect to the covariates.
	Establishing and finding reasonable conditions for their equivalence is indeed not trivial as shown in Theorem \ref{thm:hierarchy consistency} below.

	\section{Main Result and Discussion}
	\label{sec:main result}
	
	To present our main Theorem \ref{thm:hierarchy consistency}, we introduce and discuss the following two assumptions.
	\begin{assumption}\label{ass:separability}
		For all $X\in\X$, the map $m(X,\cdot)\colon\Theta\to\Xi$ is surjective almost surely. 
		For all $(Y,X)\in\Z$ the conditional expectation $\E\big[\rho\big(Y,m(X,\theta)\big)\big|X\big]$ is continuous in $\theta$ almost surely.
	\end{assumption}
	The surjectivity in Assumption \ref{ass:separability} can usually be fulfilled by a sensible choice of $\Xi$, and the smoothness condition on the expected loss is standard in the literature, see e.g., \citet[Section 2.3]{NeweyMcFadden1994}.
	
	\begin{assumption}\label{ass:reweighting}
		For any $Z = (Y,X)\in\Z$ and any event $A\in \sigma(X)$ with positive probability $\P(A)>0$, there is some $\widetilde Z\in\Z$ such that $\P\big( \widetilde Z \in B\big) = \P\big(Z\in B\,|A\big)$ for all Borel sets $B\subseteq \R^{d+p}$.
	\end{assumption}
	
	Assumption \ref{ass:reweighting} is a richness condition on the class of possible data generating processes (DGP) in $\Z$ as for any process $Z = (Y,X) \in \mathcal{Z}$, and any set $A$ with positive probability, it stipulates that $\mathcal{Z}$ is rich enough to contain a process $\widetilde Z = (\widetilde Y, \widetilde X)$ as specified in Assumption \ref{ass:reweighting}. 
	Crucially, it yields that $\P(\widetilde Y\in C\, | \widetilde X) = \P(Y\in C\, | X,A)$ for all Borel sets $C\subseteq \R$. This, together with  Assumption \ref{ass:unique model}, implies that the correctly specified parameter and hence the semiparametric model, is the same under the distributions $F_{\wt Z}$ and $F_Z$; in formulae, $\theta_0(F_{\wt Z}) = \theta_0(F_Z)$. 
	Recall that in estimation, $\Z$ captures the flexibility about the underlying and in practice unknown DGP, such that a large $\Z$ is desirable in order to obtain an estimation method which is applicable to a wide range of distributions of $Z \in \Z$.
	Thus,  Assumption \ref{ass:reweighting} intuitively means that given a certain plausible and correctly specified DGP, and given a measurable set $B\subset \R^p$ of possible values for the covariates $X$ which is attained with positive probability, i.e.\ $\P(X\in B)>0$, restricting the DGP to these values of covariates must be feasible. 
	E.g., if income $Y$ is studied in dependence of years after graduation, $X_1$, and further covariates $X_2, \ldots, X_p$, one might as well study income of persons at most 5 years after their graduation, $X_1\le 5$.
	Then, in a correctly specified model, the true but unknown parameter $\theta_0$ remains the same, no matter whether considering the whole population or only persons within 5 years after their graduation.
	Further recall that as discussed after \eqref{eqn:GPLLoss}, the characterization results for strictly consistent loss functions already rely on richness conditions on the classes of distributions.
	
	\begin{theorem}
		\label{thm:hierarchy consistency}
		Under Assumption \ref{ass:unique model} the following holds for a loss $\rho\colon \R\times \Xi\to\R$.
		\begin{enumerate}[label=\rm(\roman*)]
			\item
			If $\rho$ is (strictly) $\F_{\Y|\X}$-consistent for $\Gamma$ then it is (strictly) conditionally $\F_\Z$-model-consistent for the model $m$.
			\item
			Under Assumption \ref{ass:separability} if $\rho$ is conditionally $\F_\Z$-model-consistent for $m$, 
			there is a modification $\widetilde \F_{\Y|\X}$ of $\F_{\Y|\X}$ such that $\rho$ is $\widetilde \F_{\Y|\X}$-consistent for $\Gamma$.
			\item
			If $\rho$ is (strictly) conditionally $\F_\Z$-model consistent for $m$ then it is (strictly) unconditionally $\F_\Z$-model-consistent for $m$.
			\item
			Under Assumption \ref{ass:reweighting} if $\rho$ is (strictly) unconditionally $\F_\Z$-model-consistent for $m$ then it is also (strictly) conditionally $\F_\Z$-model-consistent for $m$.
		\end{enumerate}
	\end{theorem}
	
	Theorem \ref{thm:hierarchy consistency}, whose proof can be found in Appendix \ref{sec:proof}, provides two main implications; see Figure \ref{fig:implications} for a visualisation: 
	First, a combination of (i) and (iii) justifies the use of strictly $\F_{\Y|\X}$-consistent losses for $\Gamma$ in the context of M-estimation. This is well known in the literature, e.g.\ \citet[Section 9]{GneitingRaftery2007} describe this under the term optimum score estimation.
	The proofs of (i) and (iii) are straight-forward, and for special cases they can be found, e.g.\ in the proof of \citet[Theorem 1]{Patton2019}.
	
	\begin{figure}
		\centering 
		\scalebox{0.9}{
			\begin{tikzpicture}
				\node[state, minimum size=4cm] (q1) {\small \begin{tabular}{c} $\F_{\Y|\X}$-consistency \\ for $\Gamma$ \end{tabular}};
				\node[state, right of=q1, minimum size=4cm] (q2)  {\small \begin{tabular}{c} conditional \\ $\F_{\Z}$-consistency \\ for $m$ \end{tabular}};
				\node[state, right of=q2, minimum size=4cm] (q3) {\small \begin{tabular}{c} unconditional \\ $\F_{\Z}$-consistency  \\ for $m$ \end{tabular}};
				\draw   
				(q1) edge[->, bend left, above, line width=2pt] node{(i)} (q2)
				(q2) edge[bend left, below,  line width=2pt] node{(ii)} (q1)
				(q2) edge[bend left, above,  line width=2pt] node{(iii)} (q3)
				(q3) edge[bend left, below,  line width=2pt] node{(iv)} (q2);
			\end{tikzpicture}
		}
		\caption{A visualisation of the implications of Theorem \ref{thm:hierarchy consistency}.}
		\label{fig:implications}
	\end{figure}
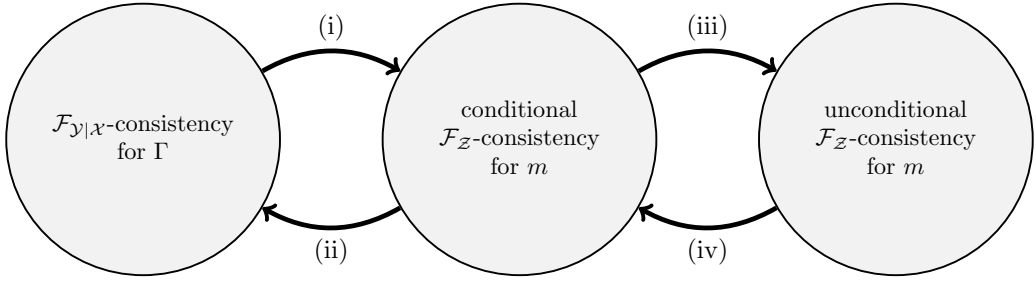
	
	Second, and more important for our purposes is the reverse implication, combining (ii) and (iv). 
	It asserts that, under appropriate assumptions, an unconditionally $\F_{\Z}$-model-consistent loss for $m$ is necessarily $\F_{\Y|\X}$-consistent for $\Gamma$. 
	Thus, exploiting known characterization results for $\F_{\Y|\X}$-consistent losses for many relevant functionals $\Gamma$, it constitutes an effective and original bound on the class of consistent M-estimators. 
	Notice that \emph{strictness} of the $\widetilde \F_{\Y|\X}$-consistent losses cannot be established in part (ii) of Theorem \ref{thm:hierarchy consistency}.
	While a stronger version of this result including the strictness would be desirable, its lack hardly diminishes the applicability of the results since characterization results for  non-strict $\F_{\Y|\X}$-consistent losses are available in the literature and these are almost as strong as the ones  for strictly consistent losses \citep{Gneiting2011}.
	E.g., we obtain all consistent losses for the mean functional when possibly non-strictly convex functions $\phi$ are used in \eqref{eqn:BregmanLoss}, and for quantiles when possibly non-strictly increasing functions $g$ are used in  \eqref{eqn:GPLLoss}.
	Further note that for practical or intuitive purposes, the technical distinction between $\F_{\Y|\X}$ and a modification $\widetilde \F_{\Y|\X}$ thereof is inessential, see Appendix \ref{sec:App1}.
	
	E.g., Theorem \ref{thm:hierarchy consistency} implies that the class of M-estimators for conditional mean models are characterized by loss functions of the form \eqref{eqn:BregmanLoss} while for conditional quantiles, a loss of the form \eqref{eqn:GPLLoss} must be used.
	This enables to find beneficial choices of $\phi$, $g$ and $\kappa$ in terms of efficiency and equivariance \citep{DFZ2020} or integrability of the loss.
	Similarly, one can leverage the flexibility of the class of consistent losses for VaR and ES provided in \cite{FisslerZiegel2016}.
	

	As of yet, implications in the direction of the points (ii) and (iv) of Theorem \ref{thm:hierarchy consistency} have only been provided for special cases or under much stronger conditions:
	First, \citet{Gourieroux1987} consider M-estimators for general semiparametric models by restricting attention to parameters identified by a set of conditional moment restrictions, as given by their equation (4.5). 
	As a consequence, their main result, Property 4.7, characterizes the functional form of the losses' first-order conditions and is hence closest to our Proposition S3 in the Supplementary Material.
	In contrast, our Theorem \ref{thm:hierarchy consistency} allows us to conveniently connect M-estimation to known classes of strictly consistent loss functions from the literature on forecast evaluation.
	\citet{Gourieroux1987} further operate under a stronger and less interpretable richness condition on the class of distributions; compare their Assumption A.7 i) to our Assumption \ref{ass:reweighting}.
	Second, \citet[Theorem 2]{Komunjer2005} shows a necessary condition for M-estimation if $\Gamma$ is some quantile.
	However, a richness condition, corresponding to our Assumption \ref{ass:reweighting}, is only assumed implicitly in the proof when quantifying over all model classes before their equation (19).
	In contrast, our Theorem \ref{thm:hierarchy consistency} rigorously shows this relation for semiparametric models for any elicitable functional.

	\section*{Acknowledgement}
	
	T.~Dimitriadis gratefully acknowledges support of the German Research Foundation (DFG) through grant number 502572912, of the Heidelberg Academy of Sciences and Humanities and the Klaus Tschira Foundation. J.~Ziegel gratefully acknowledges support of the Swiss National Science Foundation.
	We are very grateful to Jana Hlavinov\'a for a careful proofreading and valuable feedback on an earlier version of this paper.

	\section*{Supplementary Material}
	\label{SM}
	
	The Supplementary Material derives results corresponding to Theorem \ref{thm:hierarchy consistency} for zero  (Z) estimators and identification functions.

	\appendix
	\section*{Appendix}
	\section{Technical Details on the Class of Conditional Distributions}
	\label{sec:App1}
	
	$\F_{\Y|\X}$ is a collection of distributions such that for any $ (Y,X)\in\Z$ a regular version of the conditional distribution $F_{Y|X}$ is an element of $\F_{\Y|\X}$ almost surely.
	Recall that $F_{Y|X}$ is unique almost surely. 
	Therefore, $\F_{\Y|\X}$ is induced by a map $\Z\ni Z\mapsto \Omega_Z\in\{\Omega'\in \mathcal A \,|\, \P(\Omega')=1\}$ such that $\F_{\Y|\X}$ is the union of all collections $\{ F_{Y|X}(\cdot,\omega)\,|\, \omega \in \Omega_Z\}$ where $Z\in\Z$.
	Since there are several such maps $\Z\to \{\Omega'\in \mathcal A \,|\, \P(\Omega')=1\}$, 
	$\F_{\Y|\X}$ is not unique, but has several \emph{modifications} corresponding to the choices of this map $Z\mapsto \Omega_Z \in \{\Omega'\in \mathcal A \,|\, \P(\Omega')=1\}$.

	\section{Proof of Theorem \ref{thm:hierarchy consistency}}
	\label{sec:proof}

		\begin{proof}[Proof of Theorem \ref{thm:hierarchy consistency}]
			(i)
			Let $Z=(Y,X)\in\Z$ and $\theta_0 = \theta_0(F_Z)$. Suppose that $\theta\neq \theta_0$. Then $\E\big[\rho\big(Y,m(X,\theta_0)\big)\big|X\big] = \E\big[\rho\big(Y,\Gamma(F_{Y|X})\big)\big|X\big] \le \E\big[\rho\big(Y,m(X,\theta)\big)\big|X\big]$ due to the $\F_{\Y|\X}$-consistency of $\rho$. Invoking Assumption \ref{ass:unique model}, we get $\P(m(X,\theta_0)\neq m(X,\theta))>0$. Therefore, if $\rho$ is (strictly) $\F_{\Y|\X}$-consistent for $\Gamma$, it is also (strictly) $\F_\Z$-model-consistent for $m$.
			
			(ii)
			Suppose that $\rho$ is conditionally $\F_\Z$-model consistent and let $(Y,X)\in\Z$. Then for all $\theta\in \Theta$ it holds that 
			\(
			\P\Big(\E\big [\rho\big(Y, m(X,\theta_0)\big)\big|X\big]\le \E\big [\rho\big(Y, m(X,\theta)\big)\big|X\big]\Big)=1.
			\)
			Since the countable union of null sets is again a null set, this implies that 
			\(
			\P\Big(\E\big [\rho\big(Y, m(X,\theta_0)\big)\big|X\big]\le \E\big [\rho\big(Y, m(X,\theta)\big)\big|X\big] \ \forall \theta\in\Theta\cap \mathbb Q^q\Big)=1,
			\)
			where $\mathbb Q$ is the set of all rationals.
			Using the fact that $\Theta\cap \mathbb Q^q$ is dense in $\Theta$ and due to the stipulated continuity of the conditional expectations of the losses in Assumption \ref{ass:separability}, we obtain that $\P(A)=1$ where
			\begin{align*}
				A = \Big\{\omega\in\Omega\colon\E\big [\rho\big(Y, m(X,\theta_0)\big)\big|X\big](\omega)
				= \E\big [\rho\big(Y, \Gamma(F_{Y|X})\big)\big|X\big](\omega) \\
				\le \E\big [\rho\big(Y, m(X,\theta)\big)\big|X\big](\omega) \ \forall \theta\in\Theta\Big\} \in \mathcal A,
			\end{align*}
			where we also used Assumption \ref{ass:unique model}.
			Let $A'\in\mathcal A$ be the set with probability one such that $m(X(\omega),\cdot)$ is surjective for all $\omega\in A'$.
			Then for all $\omega\in A\cap A'$ we get that
			\[
			\int \rho\big(y,\Gamma(F_{Y|X}(\cdot,\omega))\big)F_{Y|X}(\diff y, \omega)
			\le \int \rho\big(y,m(X(\omega),\theta)\big)F_{Y|X}(\diff y, \omega)
			\]
			for all $\theta\in\Theta$. Finally, exploiting the surjectivity of the model, we arrive at the claim.
			Clearly, it is only possible to establish this assertion on a modification of $\F_{\Y|\X}$; see the first paragraph of Section \ref{subsec:Notation and Setting} and Appendix \ref{sec:App1}.

			(iii)
			This is a standard application of the tower property together with the positivity of the expectation.
			
			(iv)
			Assume that $\rho$ is not strictly conditionally $\F_\Z$-model-consistent for $m$. That means there exists $Z= (Y,X)\in\Z$ with true parameter $\theta_0 = \theta_0(F_Z)$ such that for some $\theta\neq \theta_0$ the event $
			A = \big\{\omega\, |\, \E\big[\rho\big(Y,m(X,\theta)\big) - \rho\big(Y,m(X,\theta_0)\big)\big| X\big](\omega) \le 0\big\}$
			has positive probability.
			Let $\wt Z= (\wt Y, \wt X)\in\Z$ be the pair given by Assumption \ref{ass:reweighting} with $A$ specified above. Then clearly
			\be{eq:proof inequ}
			\E\big[\rho\big(\wt Y,m(\wt X,\theta)\big) - \rho\big(\wt Y,m(\wt X,\theta_0)\big)\big] \\
			= \E\Big[\E\big[\rho\big(\wt Y,m(\wt X,\theta)\big) - \rho\big(\wt Y,m(\wt X,\theta_0)\big)\big|\wt X\big]\Big]  \le0.
			\ee
			This means that $\rho$ is not strictly unconditionally $\F_{\Z}$-model-consistent for $m$. The argument when we assume that $\rho$ is merely conditionally $\F_\Z$-model-consistent works analogously, where 
			we replace the inequalities in the definition of $A$ and in \eqref{eq:proof inequ} with strict inequalities.
		\end{proof}

\singlespacing
\setlength{\bibsep}{2pt plus 0.3ex}
\def\bibfont{\small}

\bibliographystyle{apalike}
\bibliography{biblio_Bridging}

\begin{thebibliography}{}

\bibitem[Barendse, 2022]{Barendse2020}
Barendse, S. (2022).
\newblock Efficiently weighted estimation of tail and interquartile
  expectations.
\newblock {\em Preprint}.
\newblock
  \href{https://dx.doi.org/10.2139/ssrn.2937665}{https://dx.doi.org/10.2139/ssrn.2937665}.

\bibitem[Bierens, 1990]{Bierens1990}
Bierens, H.~J. (1990).
\newblock A consistent conditional moment test of functional form.
\newblock {\em Econometrica}, 58(6):1443--1458.

\bibitem[Bollerslev, 1986]{Bollerslev1986}
Bollerslev, T. (1986).
\newblock Generalized autoregressive conditional heteroskedasticity.
\newblock {\em Journal of Econometrics}, 31(3):307--327.

\bibitem[Brown, 1983]{Brown1983}
Brown, B.~W. (1983).
\newblock The identification problem in systems nonlinear in the variables.
\newblock {\em Econometrica}, 51(1):175--196.

\bibitem[Chamberlain, 1987]{Chamberlain1987}
Chamberlain, G. (1987).
\newblock {Asymptotic efficiency in estimation with conditional moment
  restrictions}.
\newblock {\em Journal of Econometrics}, 34(3):305--334.

\bibitem[Dimitriadis and Bayer, 2019]{DimiBayer2019}
Dimitriadis, T. and Bayer, S. (2019).
\newblock A joint quantile and expected shortfall regression framework\protect.
\newblock {\em Electrononic Journal of Statistics}, 13(1):1823--1871.

\bibitem[Dimitriadis et~al., 2022]{DFZ_OsbandID}
Dimitriadis, T., Fissler, T., and Ziegel, J. (2022).
\newblock Osband's principle for identification functions.
\newblock {\em Preprint}.
\newblock
  \href{https://arxiv.org/abs/2208.07685}{https://arxiv.org/abs/2208.07685}.

\bibitem[Dimitriadis et~al., 2021a]{DFZ2020}
Dimitriadis, T., Fissler, T., and Ziegel, J.~F. (2021a).
\newblock The efficiency gap.
\newblock {\em Preprint}, (version v2).
\newblock
  \href{https://arxiv.org/abs/2010.14146v2}{https://arxiv.org/abs/2010.14146v2}.

\bibitem[Dimitriadis et~al., 2021b]{DGJ_2021}
Dimitriadis, T., Gneiting, T., and Jordan, A.~I. (2021b).
\newblock Stable reliability diagrams for probabilistic classifiers.
\newblock {\em Proceedings of the National Academy of Sciences},
  118(8):e2016191118.

\bibitem[Dimitriadis et~al., 2021c]{DimiPattonSchmidt2019}
Dimitriadis, T., Patton, A.~J., and Schmidt, P.~W. (2021c).
\newblock Testing forecast rationality for measures of central tendency.
\newblock {\em Preprint}.
\newblock
  \href{https://arxiv.org/abs/1910.12545}{https://arxiv.org/abs/1910.12545}.

\bibitem[Ehm et~al., 2016]{EhmETAL2016}
Ehm, W., Gneiting, T., Jordan, A., and Kr{\"u}ger, F. (2016).
\newblock {Of quantiles and expectiles: consistent scoring functions, Choquet
  representations and forecast rankings}.
\newblock {\em Journal of the Royal Statistical Society: Series B (Statistical
  Methodology)}, 78(3):505--562.

\bibitem[Fissler and Ziegel, 2016]{FisslerZiegel2016}
Fissler, T. and Ziegel, J.~F. (2016).
\newblock {Higher order elicitability and Osband's principle}.
\newblock {\em Annals of Statistics}, 44(4):1680--1707.

\bibitem[Fissler and Ziegel, 2019]{FisslerZiegel2019}
Fissler, T. and Ziegel, J.~F. (2019).
\newblock Order-sensitivity and equivariance of scoring functions.
\newblock {\em Electronic Journal of Statistics}, 13(1):1166--1211.

\bibitem[Gneiting, 2011a]{Gneiting2011}
Gneiting, T. (2011a).
\newblock Making and evaluating point forecasts.
\newblock {\em Journal of the American Statistical Association},
  106(494):746--762.

\bibitem[Gneiting, 2011b]{Gneiting2011b}
Gneiting, T. (2011b).
\newblock Quantiles as optimal point forecasts.
\newblock {\em International Journal of Forecasting}, 27(2):197--207.

\bibitem[Gneiting and Raftery, 2007]{GneitingRaftery2007}
Gneiting, T. and Raftery, A. (2007).
\newblock Strictly proper scoring rules, prediction, and estimation.
\newblock {\em Journal of the American Statistical Association},
  102(477):359--378.

\bibitem[Gourieroux et~al., 1987]{Gourieroux1987}
Gourieroux, C., Monfort, A., and Renault, E. (1987).
\newblock Consistent {M-estimators} in a semi-parametric model.
\newblock {\em CEPREMAP Working Paper 8720}.
\newblock
  \href{http://www.cepremap.fr/depot/couv_orange/co8720.pdf}{http://www.cepremap.fr/depot/couv\_orange/co8720.pdf}.

\bibitem[Hansen, 1982]{Hansen1982}
Hansen, L.~P. (1982).
\newblock Large sample properties of generalized method of moments estimators.
\newblock {\em Econometrica}, 50(4):1029--54.

\bibitem[Huber, 1967]{Huber1967}
Huber, P.~J. (1967).
\newblock {T}he behavior of maximum likelihood estimates under nonstandard
  conditions.
\newblock In {\em Proceedings of the Fifth Berkeley Symposium on Mathematical
  Statistics and Probability}, pages 221--233. Berkeley: University of
  California Press.

\bibitem[Kemp and Silva, 2012]{Kemp2012}
Kemp, G.~C. and Silva, J.~S. (2012).
\newblock Regression towards the mode.
\newblock {\em Journal of Econometrics}, 170(1):92--101.

\bibitem[Koenker and Bassett, 1978]{Koenker1978}
Koenker, R. and Bassett, G. (1978).
\newblock {R}egression quantiles.
\newblock {\em Econometrica}, 46(1):33--50.

\bibitem[Komunjer, 2005]{Komunjer2005}
Komunjer, I. (2005).
\newblock Quasi-maximum likelihood estimation for conditional quantiles.
\newblock {\em Journal of Econometrics}, 128(1):137--164.

\bibitem[Komunjer, 2012]{Komunjer2012}
Komunjer, I. (2012).
\newblock Global identification in nonlinear models with moment restrictions.
\newblock {\em Econometric Theory}, 28(4):719--729.

\bibitem[Murphy and Daan, 1985]{MurphyDaan1985}
Murphy, A.~H. and Daan, H. (1985).
\newblock Forecast evaluation.
\newblock In Murphy, A.~H. and Katz, R.~W., editors, {\em Probability,
  {S}tatistics and {D}ecision {M}aking in the {A}tmospheric {S}ciences}, pages
  379--437. Westview Press, Boulder, Colorado.

\bibitem[Newey, 1985]{Newey1985}
Newey, W.~K. (1985).
\newblock Maximum likelihood specification testing and conditional moment
  tests.
\newblock {\em Econometrica}, 53(5):1047--1070.

\bibitem[Newey, 1990]{Newey1990}
Newey, W.~K. (1990).
\newblock Semiparametric efficiency bounds.
\newblock {\em Journal of Applied Econometrics}, 5(2):99--135.

\bibitem[Newey, 1993]{Newey1993}
Newey, W.~K. (1993).
\newblock Efficient estimation of models with conditional moment restrictions.
\newblock In Maddala, G., Rao, C., and Vinod, H., editors, {\em {Handbook of
  Statistics, Volume 11: Econometrics}}.

\bibitem[Newey and McFadden, 1994]{NeweyMcFadden1994}
Newey, W.~K. and McFadden, D. (1994).
\newblock {L}arge sample estimation and hypothesis testing.
\newblock In Engle, R.~F. and McFadden, D., editors, {\em {Handbook of
  Econometrics}}, volume~4, chapter~36, pages 2111--2245. Elsevier.

\bibitem[Nolde and Ziegel, 2017]{NoldeZiegel2017}
Nolde, N. and Ziegel, J.~F. (2017).
\newblock {Elicitability and backtesting: Perspectives for banking regulation}.
\newblock {\em Annals of Applied Statistics}, 11(4):1833--1874.

\bibitem[Patton et~al., 2019]{Patton2019}
Patton, A.~J., Ziegel, J.~F., and Chen, R. (2019).
\newblock Dynamic semiparametric models for expected shortfall (and
  value-at-risk).
\newblock {\em Journal of Econometrics}, 211(2):388 -- 413.

\bibitem[Roehrig, 1988]{Roehrig1988}
Roehrig, C.~S. (1988).
\newblock Conditions for identification in nonparametric and parametric models.
\newblock {\em Econometrica}, 56(2):433--447.

\bibitem[Rothenberg, 1971]{Rothenberg1971}
Rothenberg, T.~J. (1971).
\newblock Identification in parametric models.
\newblock {\em Econometrica}, 39(3):577--591.

\bibitem[Savage, 1971]{Savage1971}
Savage, L.~J. (1971).
\newblock Elicitation of personal probabilities and expectations.
\newblock {\em Journal of the American Statistical Association},
  66(336):783--801.

\end{thebibliography}

\newpage
\onehalfspacing
\setcounter{page}{1}
\setcounter{footnote}{0}
\begin{center}
	SUPPLEMENTARY MATERIAL FOR   \vspace{10pt} \\
	{\Large\bf {Characterizing M-estimators} \vspace{10pt} }\\
	Timo Dimitriadis, Tobias Fissler and Johanna Ziegel \\
	\today \\ 
\end{center}

\renewcommand{\thesection}{S\arabic{section}}   
\renewcommand{\thepage}{S.\arabic{page}}  
\renewcommand{\thetable}{S\arabic{table}}   
\renewcommand{\thefigure}{S\arabic{figure}}   
\renewcommand{\thetheorem}{S\arabic{theorem}}   
\renewcommand{\thedefinition}{S\arabic{definition}}   
\renewcommand{\theproposition}{S\arabic{proposition}}   
\renewcommand{\theexample}{S\arabic{example}}   
\renewcommand{\theequation}{S\arabic{equation}}   

\setcounter{section}{0}
\setcounter{table}{0}
\setcounter{figure}{0}
\setcounter{theorem}{0}
\setcounter{definition}{0}
\setcounter{proposition}{0}
\setcounter{example}{0}
\setcounter{equation}{0}

This supplement  provides results on linking zero (Z) estimation to identification functions from forecast evaluation.	It further illustrates why a full analogon of Theorem \ref{thm:hierarchy consistency} in the main text is not achievable by providing corresponding counterexamples.
Section \ref{sec:intro} introduces the setup and notation and Section \ref{sec:main result_Zest} provides results and a discussion thereof. All proofs are given in Section \ref{sec:proofs}.

\section{Setup, Notation and Definitions}
\label{sec:intro}


Recall the estimation problem for semiparametric models of a general functional $\Gamma$ as outlined in the main text, where we shall use the same notation.
A correctly specified parametric model $m(X,\theta)$ satisfies
$\Gamma(F_{Y|X}) = m(X,\theta_0) \in \Xi \subseteq \mathbb{R}^k$ for some unique parameter $\theta_0\in\Theta \subseteq \mathbb{R}^q$.
The task is to estimate the parameter $\theta_0$ based on data $(Y_t,X_t)$, $t=1, \ldots, N$. 
A standard alternative to M-estimation is (zero) Z-estimation
\begin{align}
	\label{eq:Z-est}
	\widehat \theta_{Z,T} = \argmin_{\theta \in\Theta} \Big\|\frac{1}{T} \sum_{t=1}^T 
	\psi (Y_t , X_t,\theta)
	\Big\|^2,
\end{align}
based on some $s$-dimensional functions $\psi(Y_t , X_t,\theta)$, that could also be time-varying.
The latter are often called \textit{moment conditions} or \textit{identification functions} for $\theta_0$, satisfying the \emph{strict unconditional identification} condition
\begin{align}
	\label{eq:model ind}
	\Big(\E\big[\psi (Y_t , X_t,\theta)\big]=0 
	\quad \Longleftrightarrow \quad \theta=\theta_0\Big) \qquad \text{for all } \theta\in\Theta \quad \text{for all } t\in\N.
\end{align}
Motivated by the zero-condition in \eqref{eq:model ind}, the estimator in \eqref{eq:Z-est} is called Z-estimator.
We restrict attention to the \textit{exactly identified} case of $s=q$, implying as many moment conditions as model parameters.
Extensions to the case $s > q$ are possible in the framework of generalized method of moments (GMM) estimation \citep{Hansen1982, NeweyMcFadden1994}.


Similar to the strict model-consistency of loss functions, the strict unconditional identification \eqref{eq:model ind} is the core condition for the consistency of the Z-estimator $\widehat \theta_{Z,T}$ and the choice of $\psi$ is the key ingredient of Z-estimation, see e.g., \cite{Gourieroux1987}, \cite{Newey1990}, \cite{NeweyMcFadden1994}.
We continue to formally introduce the notions of identification functions from the literature on forecast evaluation, where these functions are deployed to check (conditional) calibration of forecasts \citep{NoldeZiegel2017, DimiPattonSchmidt2019}, akin to a goodness-of-fit test. 


\begin{definition}[Identification function and identifiability]
	\label{defn:identifiability}
	A  map $\ph\colon \R^d \times \Xi\to\R^k$ is called an \emph{$\F$-identification function} for a functional $\Gamma\colon \F\to\Xi \subseteq \R^k$ if
	$\ph(\cdot, \xi)$ is $F$-integrable for all $F\in\F$ and for all $\xi\in\Xi$, and if
	\[
	\int \ph \big(y,\Gamma(F)\big) \dint F(y) =0 \qquad \text{for all } F\in\F.
	\]
	If additionally 
	\[
	\left(\int \ph \big(y,\xi \big) \dint F(y)= 0 \ \implies \  \xi = \Gamma(F)\right) \qquad \text{for all } F\in\F, \  \text{for all } \xi \in \Xi,
	\]
	it is a \emph{strict $\F$-identification function} for $\Gamma$.
	A functional $\Gamma\colon \F\to\Xi$ is called \emph{identifiable} if there is a strict $\F$-identification function for it.
\end{definition}

In this definition, we restrict attention to $k$-dimensional identification functions.
This is motivated by the characterization result of \cite{DFZ_OsbandID}, who show that, given a strict identification function $\ph(y,\xi)$, the whole class of strict identification functions is given by the set
\begin{align}
	\label{eq:all identification functions}
	\big\{ h(\xi)\ph(y,\xi)\,|\, h\colon \Xi \to\R^{k\times k}, \ \det (h(\xi))\neq0 \ \text{for all }\xi\in\Xi \big\}.
\end{align}
This implies that identification functions of lower dimension than $k$ cannot be strict, while ones with dimension greater $k$ contain redundancies; see \citet[Remark 6]{DFZ_OsbandID} for a detailed discussion.

For estimation of semiparametric models, and similar to Definition 1 we introduce an unconditional and a conditional notion of identification functions for models, 
where the latter coincides with the notion of conditional moment conditions (or restrictions), as given, e.g.\ in \cite{Newey1993} and the references therein.

For this, we make use of the observation that by Assumption \ref{ass:unique model}, the parameter $\theta_0 = \theta_0(F_Z)$ can be interpreted as a functional $\theta_0 \colon \F_\Z\to\Theta\subseteq \R^q$  by mapping a distribution $F_Z\in\F_\Z$ to the unique $\theta_0(F_Z)\in\interior(\Theta)$ such that (2) is satisfied.


\begin{definition}
	\label{defn:H-identification}
	Under Assumption \ref{ass:unique model}, let $\theta_0\colon\F_\Z\to\Theta\subseteq \R^q$ be the functional given by \eqref{eq:unique model}. Let $\psi\colon \R^d\times \R^p\times \Theta\to\R^q$ be a function such that $\E\|\psi(Y,X, \theta)\|_1<\infty$ for all $(Y,X)\in\Z$ and for all $\theta\in\Theta$. 
	\begin{enumerate}[label=(\roman*)]
		\item
		The function $\psi$ is an \emph{unconditional $\F_\Z$-identification function} for $\theta_0\colon\F_\Z\to\Theta$ if 
		\[
		\E\big[\psi(Y,X,\theta_0(F_Z))\big] = 0 \qquad  \text{for all } Z=(Y,X)\in\Z\,.
		\]
		It is a \emph{strict unconditional $\F_\Z$-identification function} for $\theta_0$ if additionally 
		\[
		\Big(\E[\psi(Y,X,\theta)] = 0 \ \implies \  \theta = \theta_0(F_Z)\Big) \qquad  \text{for all } Z=(Y,X)\in\Z, \  \text{for all } \theta\in \Theta\,.
		\]
		\item
		The function $\psi$ is a \emph{conditional $\F_\Z$-identification function} for $\theta_0\colon\F_\Z\to\Theta$ if 
		\[
		\E\big[\psi(Y,X,\theta_0(F_Z))\big  |X\big] = 0 \quad \as
		\qquad  \text{for all } Z=(Y,X)\in\Z\,.
		\]
		It is a \emph{strict conditional $\F_\Z$-identification function} for $\theta_0$ if additionally 
		\[
		\Big(\E\big[\psi(Y,X,\theta)\big  |X\big] = 0 \quad \as \ \implies \ \theta = \theta_0(F_Z)\Big)
		\qquad  \text{for all } Z=(Y,X)\in\Z, \  \text{for all }  \theta\in \Theta\,.
		\]
	\end{enumerate}
\end{definition}


\section{Main result and discussion}
\label{sec:main result_Zest}

The following Proposition gives counterparts of Theorem \ref{thm:hierarchy consistency} (i) and (ii), with similar attenuations with respect to the \emph{strictness} as in Theorem \ref{thm:hierarchy consistency} (ii).


\begin{proposition}
	\label{prop:hierarchy ident}
	Under Assumption \ref{ass:unique model}, the following implications hold for $\ph\colon\R^d\times\Xi\to\R^k$:
	\begin{enumerate}[label=(\roman*)]
		\item
		If $\ph$ is a (strict) $\F_{\Y|\X}$-identification function for $\Gamma$ then $\R^d\times \R^p \times \Theta\ni (y,x,\theta) \mapsto \ph\big(y,m(x,\theta)\big)$ is a (strict) conditional $\F_\Z$-identification function for $\theta_0$.
		\item
		If $(y,x,\theta) \mapsto \ph\big(y,m(x,\theta)\big)$ is a conditional $\F_\Z$-identification function for $\theta_0$ 
		then $\ph$ is an $\F_{\Y|\X}$-identification function for $\Gamma$.
	\end{enumerate}
\end{proposition}

While counterparts to the conclusions of Theorem \ref{thm:hierarchy consistency} (iii) and (iv), connecting the conditional with the unconditional notion, would be desirable for identification functions, they seem out of reach in the general case.
To arrive at a counterpart of Theorem \ref{thm:hierarchy consistency} (iii) note that, by the tower property, any (strict) conditional $\F_\Z$-identification for $\theta_0$ is an unconditional $\F_\Z$-identification function for $\theta_0$. However, it generally fails to be strict; see Example \ref{exmp:counter-example linear model} below and the results of \citet{NeweyMcFadden1994}, \citet{Rothenberg1971}, \citet{Brown1983}, \citet{Roehrig1988}, and \citet{Komunjer2012} among many others.

Henceforth,  and similar to \citet[Equation (4.5)]{Gourieroux1987}, we restrict attention to conditional moment conditions of the form
\begin{align}
	\label{eq:psi_A definition}
	\psi_A(Y,X,\theta) = A(X, \theta) \ph\big(Y,m(X,\theta)\big),
\end{align}
where $\ph$ is a strict $\F_{\Y|\X}$-identification function for $\Gamma$, and $A(X,\theta)$ is a $(q\times k)$ instrument matrix.
This construction is motivated by the well-known fact that $\E\big[\ph\big(Y,m(X,\theta)\big)\big|X\big] =0$ $\as$ is equivalent to
\begin{align}
	\label{eq:conditional expectation}
	\E\big[a(X)^\intercal \ph\big(Y,m(X,\theta)\big)\big] =0 \qquad \text{for all measurable }\ a\colon \R^p\to\R^k,
\end{align}
such that $\E\big[\big|a(X)^\intercal \ph\big(Y,m(X,\theta)\big)\big|\big]<\infty$, where the functions $a(X) = a(X,\theta)$ can also depend on $\theta$. 
To render the equivalence in \eqref{eq:conditional expectation} statistically feasible, we reduce the number of functions $a(X,\theta)^\intercal$ to be finite, resulting in the instrument matrix $A(X,\theta)\in\R^{s\times k}$ in \eqref{eq:psi_A definition}.
Also see \cite{Newey1985}, \cite{Newey1990}, \cite{Bierens1990}, \cite{NoldeZiegel2017} for similar constructions in estimation and testing based on conditional moment restrictions.

\begin{rem}
	\label{rem:ChoicePhi}
	In light of the characterization result of \cite{DFZ_OsbandID}, this has the advantage that the choice of the identification function $\ph$ in \eqref{eq:psi_A definition} is actually irrelevant. Indeed, suppose one considers $\ph'\colon\R^d\times\Xi\to\R^k$ rather than $\ph$ in \eqref{eq:psi_A definition}. This means there is a matrix-valued function $h\colon\interior(\Xi)\to \R^{k\times k}$ of full rank such that $\ph'\big(y,m(x,\theta)\big) = h\big(m(x,\theta)\big) \ph\big(y,m(x,\theta)\big)$. 
	Then we can use the matrix $A'(x,\theta) = A(x,\theta) \big(h(m(x,\theta))\big)^{-1}$ such that
	\(
	A'(x,\theta)\ph'\big(y,m(x,\theta)\big) = A(x,\theta)\ph\big(y,m(x,\theta)\big).
	\)
	Consequently, it is no loss of generality to fix a certain strict $\F_{\Y|\X}$-identification function $\ph$ for $\Gamma$ since the remaining flexibility can always be captured through the choice of the instrument matrix $A$. 
\end{rem}



The following Proposition \ref{prop:primitive cond} provides a sufficient condition on the instrument matrix $A(X,\theta)$ such that $A(X,\theta)\ph\big(Y,m(X,\theta)\big)$ becomes a strict identification function for $\theta_0$.

\begin{proposition}
	\label{prop:primitive cond}
	Under Assumption \ref{ass:unique model}, let $\ph\colon \R^d\times \R^k\to\R^k$ be a strict $\F_{\Y|\X}$-identification function. Let $A\colon\R^p\times \Theta\to \R^{q\times k}$ be an instrument matrix such that
	the matrix
	$\E\big[A(X, \theta) D(X, \theta')\big]$
	has full rank, where
	\(
	D(X, \theta') = \nabla_{\theta} \E\big[\ph\big(Y, m(X,\theta)\big) |X\,\big] \,\big|_{\theta = \theta'}
	\)
	for all $(Y,X)\in\Z$ and for all $\theta, \theta'\in\Theta$ such that there is a $\lambda\in[0,1]$ with $\theta'=(1-\lambda)\theta_0 + \lambda \theta$. 
	Then $A(x,\theta) \ph\big(y,m(x,\theta)\big)$ is a strict unconditional $\F_\Z$-identification function for $\theta_0$.
\end{proposition}

The following proposition takes the angle of `reverse engineering' establishing a counterpart to Theorem \ref{thm:hierarchy consistency} (iv): When an unconditional strict identification function is of the form \eqref{eq:psi_A definition} it establishes a sufficient condition on the instrument matrix $A$ to ensure that $\ph$ is a conditional strict identification function.

\begin{proposition}
	\label{prop:hierarchy idenfitication}
	Suppose that $q\ge k$ and that Assumptions \ref{ass:unique model} and \ref{ass:reweighting} hold. Moreover, assume that for all $Z=(Y,X)\in\Z$ with correctly specified parameter $\theta_0(F_Z)$ the map $A\colon\R^p\times \Theta\to\R^{q\times k}$ satisfies 
	\be{eq:full rank}
	\P\big(\rank\big(A(X,\theta_0(F_Z))\big) = k\big) = 1.
	\ee
	If
	$\psi_A\colon \R^d\times \R^p \times \Theta \to\R^q$, $\psi_A(y,x,\theta) = A(x,\theta) \psi (y,x,\theta )$
	is a strict unconditional $\F_\Z$-identification function for $\theta_0\colon\F_\Z\to\Theta$, then $ \psi\colon \R^d\times \R^p \times\Theta\to\R^k$ is a strict conditional $\F_\Z$-identification function for $\theta_0$.
\end{proposition}

The following Example \ref{exmp:counter-example linear model} illustrates possible choices of instrument matrices in a linear mean regression context.
In particular, it shows that---in certain situations---we could relax the assumption that $A(X,\theta_0)$ needs to have full rank almost surely. However, relaxing the assumptions of Proposition \ref{prop:hierarchy idenfitication} does not seem to be a fruitful direction from our point of view, because one would need to tailor the relaxed assumptions almost on a case by case basis.

\begin{example}
	\label{exmp:counter-example linear model}
	Let $d=k=1$, $q=p\ge1$ and $\X$ be such that $\E[X X^\intercal]$ has full rank for all $X\in\X$. Then define
	\[
	\Z = \{(m(X,\theta_0) + \eps, X)\,|\, X\in\X, \ \theta_0 \in\Theta = \R^q, \ \E[\eps|X]=0 \},
	\]
	where $m(X,\theta) = X^\intercal\theta$. 
	Clearly,
	\(
	\Gamma(F_{Y|X}):= \int y \dint F_{Y|X}(y) = X^\intercal\theta_0\,.
	\)
	The condition that $\E[X X^\intercal]$ has full rank implies that the model $m$ uniquely identifies the parameter $\theta_0$ such that Assumption \ref{ass:unique model} holds. 
	Indeed, for any $\theta'\neq \theta$ it holds that 
	\(
	0< (\theta - \theta')^\intercal \E\big[XX^\intercal \big](\theta - \theta') = \E \big[\|m(X,\theta) - m(X,\theta')\|^2\big]\,.
	\)
	Employing the canonical identification function for the mean, we obtain a strict conditional $\F_{\Z}$-identification function $\ph\big(Y,m(X,\theta)\big) = m(X,\theta) - Y = X^\intercal\theta - Y$. Indeed $\E\big[\ph\big(Y,m(X,\theta)\big) |X\big] = X^\intercal(\theta - \theta_0)$. If $\theta\neq \theta_0$ and $X^\intercal(\theta - \theta_0)$ were zero $\as$, then $\E[X X^\intercal](\theta-\theta_0)=0$, violating the full rank property of $\E[X X^\intercal]$. However, $\ph\big(Y,m(X,\theta)\big)$ is in general not a \emph{strict unconditional} $\F_\Z$-identification function for $\theta_0$. It could be the case, for example, that there is an $X\in\X$ with $\E[X]=0$ such that we obtain $\E\big[\ph\big(Y,m(X,\theta)\big)\big] = \E[X^\intercal(\theta - \theta_0)]=0$ for all $\theta$.
	
	Now, choosing an instrument matrix $A(X,\theta)$ such that $\E[A(X,\theta) X^\intercal]$ has full rank for all $\theta$ is a sufficient and necessary condition to ensure that $\psi_A(Y,X,\theta) = A(X,\theta) \ph\big(Y,m(X,\theta)\big)$ is a strict \emph{unconditional} $\F_\Z$-identification function for $\theta_0$. Indeed, we obtain
	\(
	\E[\psi_A(Y,X,\theta)] = \E\big[A(X,\theta) X^\intercal\big](\theta - \theta_0).
	\)
	In particular, the choice $A(X, \theta) = X$ yields a strict unconditional $\F_\Z$-identification function for $\theta_0$.
	\qed
\end{example}

\section{Proofs}
\label{sec:proofs}

\begin{proof}[Proof of Proposition \ref{prop:hierarchy ident}]
	Part (i) is a direct application of the definitions, using similar arguments to the ones in the proof of Theorem \ref{thm:hierarchy consistency} (i).
	For part (ii) we have under Assumption \ref{ass:unique model} that 
	$0= \E\big[\ph\big(Y,m(X,\theta_0)\big)\big|X\big] = \E\big[\ph\big(Y,\Gamma(F_{Y|X})\big)\big|X\big]$. 
\end{proof}

\begin{proof}[Proof of Proposition \ref{prop:primitive cond}]
	The tower property implies $\E\big[A(X,\theta_0) \ph\big(Y,m(X,\theta_0)\big)\big]=0$. For $\theta \neq \theta_0$ the mean value theorem yields
	\begin{align*}
		&\E\big[\ph\big(Y, m(X,\theta)\big) |X\,\big] 
		= \E\big[\ph\big(Y, m(X,\theta)\big) |X\,\big] - \E\big[\ph\big(Y, m(X,\theta_0)\big) |X\,\big] \\
		&= \nabla_{\theta} \E\big[\ph\big(Y, m(X,\theta)\big) |X\,\big] \,\big|_{\theta = \theta'} (\theta - \theta_0)
		= D(X,\theta') (\theta - \theta_0)
	\end{align*}
	Therefore 
	$\E\big[A(X, \theta) \ph\big(Y, m(X,\theta)\big) \big] 
	= \E\big[A(X, \theta) D(X,\theta') \big] (\theta - \theta_0)\neq 0$.
\end{proof}

\begin{proof}[Proof of Proposition \ref{prop:hierarchy idenfitication}]
	Assume that $\psi$ is not a strict conditional $\F_\Z$-model-identification function for $\theta_0$. That means there is $Z= (Y,X)\in\Z$ with true model parameter $\theta_0 = \theta_0(F_Z)$ such that 
	\be{eq:violation_new}
	\P\big(\E[\psi(Y,X,\theta_0)|X]\neq0\big)>0
	\ee
	or
	\be{eq:violation2_new}
	\exists \theta'\neq\theta_0\colon \E[\psi(Y,X,\theta')|X]=0\quad \as
	\ee
	If \eqref{eq:violation2_new} holds, then we can directly apply the tower property to obtain that for some $\theta' \neq\theta_0$ we have
	\(
	\E[\psi_A(Y,X,\theta')] = \E \big[ A(X,\theta') \E[\psi(Y,X,\theta')|X] \big] =0,
	\)
	which means that $\psi_A$ is not a strict unconditional $\F_\Z$-identification function for $\theta_0$.
	Now, we assume that \eqref{eq:violation_new} holds. Then using \eqref{eq:full rank} we can conclude that 
	\(
	\P\big(\E[\psi_A(Y,X,\theta_0)|X]\neq0\big) 
	= \P\big(A(X,\theta_0) \E[\psi(Y,X,\theta_0)|X]\neq0\big) 
	\ge \P\big(\{\rank(A(X,\theta_0))=k\}\cap \{\E[\psi(Y,X,\theta_0)|X]\neq0\}\big) >0.
	\)
	Then, we can again argue that there exists a component $j\in\{1,\ldots, q\}$ such that 
	$\P\big(\E[\psi_{A,j}(Y,X,\theta_0)|X]<0\big)>0$ or 
	$\P\big(\E[\psi_{A,j}(Y,X,\theta_0)|X]>0\big)>0$,
	and we continue as in the proof of Theorem \ref{thm:hierarchy consistency} (iv).
\end{proof}

\end{document}